\documentclass[12pt]{article}
\usepackage{amsmath,amsfonts,amssymb,amsthm,latexsym,pb-diagram}
\newtheorem{theorem}{Theorem}[section]

\newtheorem{definition}[theorem]{Definition}
\newtheorem{corollary}[theorem]{Corollary}
\newtheorem{lemma}[theorem]{Lemma}
\newtheorem{example}[theorem]{Example}
\newtheorem{remark}[theorem]{Remark}
\newcommand{\R}{{\mathbb{R}}}

\newcommand{\Z}{{\mathbb{Z}}}
\newcommand{\N}{{\mathbb{N}}}
\DeclareMathOperator{\InTR}{Int}
\begin{document}
\begin{center}
\textbf{WEAKLY SYMMETRICALLY CONTINUOUS FUNCTIONS}
\end{center}
\begin{center}
Prapanpong Pongsriiam $^\dag$\footnote{{\tt Corresponding author email: prapanpong@gmail.com}} and Teraporn Thongsiri$^\dag$\\

\vskip.5cm

\centerline {$^\dag$Department of Mathematics, Faculty of Science, Silpakorn University, }
\centerline {Nakhon Pathom, 73000, Thailand} 
\centerline {e-mail : {\tt prapanpong@gmail.com}}
\centerline {e-mail : {\tt teraporn.thongsiri@gmail.com}}\end{center}

\begin{center}
\textbf{Abstract}
\end{center}
We extend the definition of weak symmetric continuity to be applicable for functions defined on any nonempty subset of $\R$. Then we investigate basic properties of weakly symmetrically continuous functions and compare them with those of symmetrically continuous functions and weakly continuous functions. Several examples are also given.
\vspace{0.5cm}

\section{Introduction and Preliminaries}

There are many types of generalized continuities, three of which will be discussed in this article. Throughout, let $A$ be a nonempty subset of $\R$ and $a \in A$.
\begin{itemize}
\item[(1)] A function $f:A\rightarrow\R$ is said to be symmetrically continuous at $a$ if
\begin{align*}
	\forall \varepsilon >0 \exists\delta >0 \forall h \in \R, \left|h\right|<\delta\wedge a+h, a-h \in A \Rightarrow \left|f(a+h)-f(a-h)\right| < \varepsilon.
 \end{align*}
\item[(2)]  Let $L_a(A)$ and $U_a(A)$ denote respectively the set of all strictly increasing sequence in $A$ converging to $a$ and the set of all strictly decreasing sequence in $A$ converging to $a$. 
A function $f:A\rightarrow\R$ is said to be weakly continuous at $a$ if the following statements hold:
\begin{itemize}
\item[(2.1)]  if $L_a(A) \neq \emptyset$, then there is $\left(x_n\right) \in L_a(A)$ such that $\lim_{n \rightarrow \infty} f\left(x_n\right) = f(a)$, and 
\item[(2.2)]  if $U_a(A) \neq \emptyset$, then there is $\left(y_n\right) \in U_a(A)$ such that $\lim_{n \rightarrow \infty} f\left(y_n\right) = f(a)$.
\end{itemize}
\end{itemize}

We remark that the above definitions were first given for functions defined on $\R$ and were later extended \cite{Weak, R_Sym, Sym} for functions defined on any nonempty subset of $\R$.
Similarly, the definition of weak symmetric continuity is given for functions defined on $\R$ as follows \cite{NS, Def_Weak_Sym} 

\begin{itemize}
\item[(3)]  A function $g:\R \rightarrow \R$ is said to be weakly symmetrically continuous at $a \in \R$ if there exists a sequence $\left( h_n\right)$ of positive real numbers converging to $0$ such that $\lim_{n\rightarrow \infty} f(a+h_n) - f(a-h_n) = 0$. 
\end{itemize}

In this article, we will extend (3) so that it is applicable for functions defined on a nonempty subset of $\R$. Let $S_a(A)$ be the set of all sequences $\left(h_n \right)$ of positive real numbers converging to $0$ such that $a+h_n$ and $a-h_n$ are in $A$ for every $n \in \N$. Then weak symmetric continuity can be extended as follows: 

\begin{definition} \label{Def_WSC}
A function $f:A\rightarrow\R$ is said to be weakly symmetrically continuous at $a$ if $S_a(A) \neq \emptyset$ implies that there exists a sequence $\left( h_n\right) \in S_a(A)$ such that $\lim_{n\rightarrow \infty} f(a+h_n) - f(a-h_n) = 0$.
\end{definition}

Our purpose is to obtain basic properties of weakly symmetrically continuous functions defined on a nonempty subset of $\R$ and compare them with those of weakly continuous functions and symmetrically continuous functions. We will give basic relations between these three types of continuities in Section 2. Then we will investigate combinations of these functions in Section 3. Then a study of sequences of weakly symmetrically continuous functions will be given in Section 4. We refer the reader to \cite{Be,Da,Do,H,Ko,Ma,Ok,Uni,Pon,Ros,S,S2} for other results concerning symmetric continuity of functions. We end this section by giving some results that will be used later.

\begin{lemma} {\rm\cite[Theorem 4.1]{Sym}} \label{LimitSym}
Let $f:A\rightarrow\R$ and let $a \in A$. Then $f$ is symmetrically continuous at $a$ if and only if for every sequence $\left(h_n\right)$ in $\R$ with $a+h_n,a-h_n \in A$ for every $n \in \N$, $\lim_{n\rightarrow \infty} h_n = 0$ implies that $\lim_{n\rightarrow \infty} f\left(a+h_n\right) - f\left(a-h_n\right) = 0$. 
\end{lemma}

\begin{lemma} {\rm\cite[Proposition~2.2]{Sym}} \label{Prop2.2}
Let $f:A\rightarrow\R$ and $a \in \InTR{A}$. Assume that $\lim_{x \rightarrow a^{+}} f(x) $ and $\lim_{x \rightarrow a^{-}} f(x) $ exist in $\R$. Then $f$ is symmetrically continuous at $a$ if and only if $\lim_{x \rightarrow a} f(x)$ exists.
\end{lemma}

\begin{remark} \label{remark1}
From the definition given above if $f:A\rightarrow \R$ and $a \in A$ is such that $L_a(A) = \emptyset$ and $U_a(A) = \emptyset$, then $f$ is weakly continuous at $a$. Similarly, if $S_a(A) = \emptyset$, then $f$ is symmetrically continuous and weakly symmetrically continuous at $a$. In particular, if $a$ is not a cluster point of $A$, then $f$ is weakly continuous, symmetrically continuous, and weakly symmetrically continuous at $a$.
\end{remark}

\section{Basic Relations}

Let $SC$, $WC$, and $WSC$ be, respectively, the space of symmetrically continuous functions, the space of weakly continuous functions, and the space of weakly symmetrically continuous functions defined on nonempty subsets of $\R$. In this section, we will show that the following relations hold:  

\begin{itemize}
\item[(R1)] $SC \subseteq WSC$ and $WSC \nsubseteq SC$, 
\item[(R2)] $SC \nsubseteq WC$ and $WC \nsubseteq SC$, and
\item[(R3)] $WSC \nsubseteq WC$ and $WC \nsubseteq WSC$. 
\end{itemize}

\begin{theorem} \label{SC_implie_WSC}
Let $f:A\rightarrow \R$. If $f$ is symmetrically continuous at $a$, then $f$ is weakly symmetrically continuous at $a$. In particular, if $f$ is a symmetrically continuous function, then it is a weakly symmetrically continuous function.
\end{theorem}

\begin{proof}
This follows immediately from Definition \ref{Def_WSC} and Lemma \ref{LimitSym}.
\end{proof}

Next we give various examples to show a clearer picture of each type of continuities and to verify the relations (R1), (R2) and (R3). 

\begin{example} \label{WSC_notim_SC}
Let $A=\left\{ \frac{1}{n} \mid n \in \Z-\left\{0\right\} \right\} \cup \left\{ 0 \right\}$. Define $f:\R\rightarrow\R$ by
$$
f(x) = \begin{cases}
0, \quad & \text{if} \; x \in A; \\
1, \quad & \text{if} \; x > 0 \; \text{and} \; x \notin A; \\
-1, \quad & \text{if} \; x < 0 \; \text{and} \; x \notin A. 
\end{cases} 
$$
Since $\left( \frac{1}{n} \right)_{n \in \N} \in S_0(\R)$ and 
\begin{align*}
	\lim_{n\rightarrow \infty} \left| f\left(0+\frac{1}{n}\right) - f\left(0-\frac{1}{n}\right)\right| = \lim_{n\rightarrow \infty} \left| f\left(\frac{1}{n}\right) - f\left( -\frac{1}{n}\right)\right| =0,
\end{align*}
 $f$ is weakly symmetrically continuous at $0$. Observe that $\left( \frac{1}{n} \right)_{n \in \N} \in U_0(\R)$, $\left(- \frac{1}{n} \right)_{n \in \N} \in L_0(\R)$, and $\lim_{n\rightarrow \infty} f\left(\frac{1}{n}\right) = f(0) = \lim_{n\rightarrow \infty} f\left(- \frac{1}{n}\right)$. So f is weakly continuous at $0$. However, $f$ is not symmetrically continuous at $0$. To see this, $f\left(0 + \frac{\sqrt{2}}{n}\right) - f\left(0 - \frac{\sqrt{2}}{n}\right) = 1-(-1) = 2$ which dose not converge to $0$ as $n \rightarrow \infty$. So the assertion follows from Lemma \ref{LimitSym}. Next, let $a \in \R-\left\{0\right\}$. It is easy to see that $\lim_{x\rightarrow a^+} f(x) = \lim_{x\rightarrow a^-} f(x)$ exist and are equal to $1$ or $-1$. By Lemma \ref{Prop2.2}, f is symmetrically continuous at $a$. By Theorem \ref{SC_implie_WSC}, f is weakly symmetrically continuous at $a$. This shows that f is a weakly symmetrically continuous function, is symmetrically continuous at every point $a \in \R-\left\{0\right\}$ but is not symmetrically continuous at $a=0$. In addition, $f$ is not a weakly continuous function. For instance, $f(1)=0$ but for every $\left( y_n \right) \in U_1(\R)$, $f\left(y_n\right)=1$ for every $n \in \N$, so $\left(f\left(y_n\right)\right)_{n \in \N}$ does not converge to $f(1)$. 
\end{example}
Next we give a function which is both weakly continuous and weakly symmetrically continuous but is not symmetrically continuous.
\begin{example} \label{WSC_notim_SC2}
Let $A=\left\{ \frac{1}{n} \mid n \in \Z-\left\{0\right\} \right\} \cup \left\{ 0 \right\}$ and $B =\left\{ \frac{\sqrt{2}}{n} \mid n \in \Z-\left\{0\right\} \right\}$. Define $f:A\cup B \rightarrow\R$ by
$$
f(x) = \begin{cases}
0, \quad & \text{if} \; x \in A; \\
1, \quad & \text{if} \; x > 0 \; \text{and} \; x \in B; \\
-1, \quad & \text{if} \; x < 0 \; \text{and} \; x \in B. 
\end{cases} 
$$
Let $a \in A \cup B - \left\{0\right\}$. Then $U_a(A \cup B)$, $L_a(A \cup B)$, and $S_a(A \cup B)$ are the empty set. As mention in Remark \ref{remark1}, we obtain that $f$ is weakly continuous at a, symmetrically continuous at a, and weakly symmetrically continuous at a. Similar to Example \ref{WSC_notim_SC}, $f$ is weakly symmetrically continuous at $0$, is weakly continuous at $0$, but is not symmetrically continuous at $0$. Therefore $f$ is both a weakly symmetrically continuous function and weakly continuous function, but is not a symmetrically continuous function.
\end{example}

\begin{example} \label{WC_notimplies_SC_WSC}
Let $A=\left\{ \frac{1}{n} \mid n \in \N \right\}$ and let $B= \left\{ -\frac{\sqrt{2}}{n} \mid n \in \N \right\}$. Define $f:\R \rightarrow \R$ by
$$
f(x) = \begin{cases}
x, \quad & \text{if} \; x \in A\cup B\cup \left\{ 0 \right\}; \\
1, \quad &  \text{if} \; x > 0 \; \text{and} \; x \notin A; \\
-1, \quad & \text{if} \;  x < 0 \; \text{and} \; x \notin B. 
\end{cases}
$$
Observe that $\left( \frac{1}{n} \right)_{n \in \N} \in U_0(\R)$, $\left( - \frac{\sqrt{2}}{n} \right)_{n \in \N} \in L_0(\R)$, and 
\begin{align*}
	\lim_{n\rightarrow \infty} f\left(\frac{1}{n}\right) = \lim_{n\rightarrow \infty} \frac{1}{n} = f\left(0\right) = \lim_{n\rightarrow \infty} -\frac{\sqrt{2}}{n} = \lim_{n\rightarrow \infty} f\left(-\frac{\sqrt{2}}{n}\right).
\end{align*}
Therefore $f$ is weakly continuous at $0$. But $f$ is not weakly symmetrically continuous at $0$. To see this, let $(h_n) \in S_0(\R)$. Let $n \in \N$. By the definition of $S_0(\R)$, $h_n > 0$. If $h_n \in A$, then $\left| f\left(0+h_n\right) - f\left(0-h_n\right) \right| = h_n+1$. If $h_n \notin A$, then $\left| f\left(0+h_n\right) - f\left(0-h_n\right) \right| \in \left\{1+h_n, 2\right\}$. In any case, $\left| f\left(0+h_n\right) - f\left(0-h_n\right) \right| \geq \min\left\{2, 1+h_n\right\}$ for every $n \in \N$. Since $1+h_n \rightarrow 1 \neq 0$, $f\left(0+h_n\right) - f\left(0-h_n\right)$ does not converge to $0$ as $n \rightarrow \infty$. 
\end{example}

The function $f$ in Example \ref{WC_notimplies_SC_WSC} is not a weakly continuous function because it is not  weakly continuous at $a=\frac{1}{2}$: for every $\left( y_n \right) \in U_{\frac{1}{2}}(\R)$, $f\left(y_n\right)=1$ for every $n \in \N$, and therefore $\left(f\left(y_n\right)\right)_{n \in \N}$ does not converge to $f\left(\frac{1}{2}\right)$. The next example gives a function which is a weakly continuous function but is not a weakly symmetrically continuous function.

\begin{example} \label{WC_notimplies_SC_WSC2}
Let $A=\left\{ \frac{1}{n} \mid n \in \N \right\}$, $B= \left\{ -\frac{\sqrt{2}}{n} \mid n \in \N \right\}$, $C= \left\{\frac{\sqrt{2}}{n} \mid n \in \N \right\}$, and $D = A\cup B\cup C \cup\left\{0\right\}$. Define $f:D\rightarrow \R$ by
$$
f(x) = \begin{cases}
x, \quad & \text{if} \; x \in A\cup B\cup \left\{ 0 \right\}; \\
1, \quad & \text{if} \; x \in C. 
\end{cases}
$$
Let $a \in D-\left\{0\right\}$. Then $U_a(D)$, $L_a(D)$, and $S_a(D)$ are the empty sets. Therefore $f$ is weakly continuous at $a$, symmetrically continuous at $a$, and weakly symmetrically continuous at $a$. Similar to Example \ref{WC_notimplies_SC_WSC}, $f$ is weakly continuous at $0$. But $f$ is not weakly symmetrically continuous at $0$. To see this, let $\left(h_n\right) \in S_0\left(D\right)$. By the definition of $S_0\left(D\right)$, $h_n > 0$ and $0+h_n, 0-h_n \in D$ for every $n \in \N$. Therefore $h_n \in C$ for every $n \in \N$. Then $f\left(0+h_n\right)-f\left(0-h_n\right) = 1 + h_n \rightarrow 1 \neq 0$ as $n \rightarrow \infty$. Therefore $f$ is a weakly continuous function but is not a weakly symmetrically continuous function. 
\end{example}

\begin{example} \label{SC_WSC_notimplies_WC}
Define $f:\R \rightarrow \R$ by
$$
f(x) = \begin{cases}
0, \quad & \text{if} \; x = 0; \\
1, \quad & \text{if} \; x \neq 0.  
\end{cases} 
$$
For each $a \in \R$, $\lim_{x\rightarrow a^+} f(x) = \lim_{x\rightarrow a^-} f(x) = 1$. So by Lemma \ref{Prop2.2} and Theorem \ref{SC_implie_WSC}, $f$ is a symmetrically continuous function and is a weakly symmetrically continuous function. For each $\left(y_n\right) \in U_0(\R)$, $f(y_n) \rightarrow 1 \neq f(0)$. So $f$ is not a weakly continuous function.
\end{example}

\begin{theorem} \label{Relation1}
The following statements hold:
\begin{itemize} 
\item[(i)] $SC \subseteq WSC$ and $SC \nsubseteq WC$,
\item[(ii)] $WSC \nsubseteq SC \cup WC$,
\item[(iii)] $WSC \cap WC \nsubseteq SC$, and 
\item[(iv)] $WC \nsubseteq WSC$,
\end{itemize}
\end{theorem}

\begin{proof} 
(i) follows from Theorem \ref{SC_implie_WSC} and Example \ref{SC_WSC_notimplies_WC}, (ii) follows from Example \ref{WSC_notim_SC}, (iii)  follows from Example \ref{WSC_notim_SC2}, and (iv) follows from Example \ref{WC_notimplies_SC_WSC2}.
\end{proof}

\begin{corollary}
The relations (R1), (R2) and (R3) hold.
\end{corollary}
\begin{proof}
(R1) follows from (i) and (ii) of Theorem \ref{Relation1}. (R2) follows from (i) and (iii) of Theorem \ref{Relation1}. (R3) follows from (ii) and (iv) of Theorem \ref{Relation1}.
\end{proof}

\section{Combinations of Functions}

Throughout, let $A$ and $B$ be nonempty subsets of $\R$ and let $a \in A$. The following statements are given in \cite{R_Sym} and \cite{Sym}. 
\begin{itemize}
\item[(1)] If $f, g:A\rightarrow \R$ are symmetrically continuous at $a$ and $c \in \R $, then $\left| f \right|$, $cf$, $f+g$, $f-g$, $\max\left\{f,g\right\}$, and $\min\left\{f,g\right\}$ are symmetrically continuous at $a$.
\item[(2)] Let $f, g:A\rightarrow \R$ be symmetrically continuous at $a$. If $f$ and $g$ are locally bounded at $a$, then $fg$ is symmetrically continuous at $a$.
\item[(3)] Let $f:A\rightarrow \R$ and $g:B\rightarrow \R$ be functions such that $f(A) \subseteq B$. If $f$ is symmetrically continuous at $a$, and $g$ is uniformly continuous on $B$, then $g\circ f$ is symmetrically continuous at $a$. 
\end{itemize}

Similar statements for weakly continuous functions are given in \cite{Weak}. In this article, we will investigate this kind of  properties for weakly symmetrically continuous functions.

\begin{theorem} \label{Constant_WS}
Let $f:A\rightarrow \R$ be weakly symmetrically continuous at $a$ and $c \in \R $. Then $\left| f \right|$ and $cf$ are weakly symmetrically continuous at $a$.
\end{theorem}

\begin{proof}
Assume that $S_{a}(A) \neq \emptyset$. Then there exists $\left(h_n\right) \in S_{a}(A)$ such that $\lim_{n \rightarrow \infty} f(a+h_n)-f(a-h_n) = 0$. 	
Since
\begin{align*}
	\left|\left| f\right|(a+h_n) - \left| f\right|(a-h_n) \right|= \left|\left| f(a+h_n) \right| - \left| f(a-h_n) \right|\right| 
	\leq \left| f(a+h_n) - f(a-h_n) \right|,
\end{align*}
we obtain that $\lim_{n \rightarrow \infty} \left| f\right|(a+h_n)  - \left| f\right|(a-h_n) = 0$.
In addition,  
\begin{align*}
	\lim_{n \rightarrow \infty} (cf)(a+h_n) - (cf)(a-h_n)  = c \lim_{n \rightarrow \infty}  f(a+h_n) - f(a-h_n)  = 0.
\end{align*} 
This shows that $\left|f\right|$ and $cf$ are weakly symmetrically continuous at $a$. 
\end{proof}

\begin{theorem} \label{Sum_WS}
Let $f:A\rightarrow \R$ be weakly symmetrically continuous at $a$ and let $g:A\rightarrow \R$ be symmetrically continuous at $a$. Then $f+g$, $f-g$, $\max\{ f,g \}$, and $\min\{ f,g \}$ are weakly symmetrically continuous at $a$.
\end{theorem}

\begin{proof}
Assume that $S_{a}(A) \neq \emptyset$. Then there exists $\left(h_n\right) \in S_{a}(A)$ such that 
$\lim_{n \rightarrow \infty} f(a+h_n)-f(a-h_n) = 0$.	
By Lemma \ref{LimitSym}, $g(a+h_n)-g(a-h_n)$ also converges to $0$. Then 
\begin{align*}
& \lim_{n \rightarrow \infty} (f+g)(a+h_n)-(f+g)(a-h_n) \\
& \qquad = \lim_{n \rightarrow \infty} f(a+h_n)-f(a-h_n) + \lim_{n \rightarrow \infty} g(a+h_n)-g(a-h_n) = 0.
\end{align*}
Similarly, $\lim_{n \rightarrow \infty} (f-g)(a+h_n)-(f-g)(a-h_n) = 0$. This shows that $f+g$ and $f-g$ are weakly symmetrically continuous at $a$.

Note that $\max\{ f,g \} = \frac{f+g+\left|f-g\right|}{2}$ and $\min\{ f,g \} = \frac{f+g-\left|f-g\right|}{2}$. 
Then
\begin{align*}
&\left|\max\{ f,g \}(a+h_n)-\max\{ f,g \}(a-h_n)\right| \\
	&\qquad \leq \left| \frac{f(a+h_n)-f(a-h_n)}{2} \right| + \left| \frac{g(a+h_n)-g(a-h_n)}{2} \right| \\ 
	&\qquad \quad + \left|\frac{\left|f-g\right|(a+h_n)-\left|f-g\right|(a-h_n)}{2} \right| \\ 
	&\qquad \leq \left| \frac{f(a+h_n)-f(a-h_n)}{2} \right| + \left| \frac{g(a+h_n)-g(a-h_n)}{2} \right| \\ 
	&\qquad \quad + \left|\frac{\left(f-g\right)(a+h_n)-\left(f-g\right)(a-h_n)}{2} \right| \\
	&\qquad \leq \left| \frac{f(a+h_n)-f(a-h_n)}{2} \right| + \left| \frac{g(a+h_n)-g(a-h_n)}{2} \right| \\ 
	&\qquad \quad + \left| \frac{f(a+h_n)-f(a-h_n)}{2} \right| + \left| \frac{g(a+h_n)-g(a-h_n)}{2} \right|  \\
	&\qquad = \left| f(a+h_n)-f(a-h_n) \right| + \left| g(a+h_n)-g(a-h_n)\right| \rightarrow 0
\end{align*}
as $n\rightarrow \infty$. Therefore $\max\{ f,g \}$ is weakly symmetrically continuous at $a$. Similarly, $\min\{ f,g \}$ is weakly symmetrically continuous at $a$.
\end{proof}

Although $f$ and $g$ are weakly symmetrically continuous at $a$, $f+g$, $f-g$, $\max\{ f,g \}$, and $\min\{ f,g \}$ may or may not be weakly symmetrically continuous at $a$. So we cannot replace the symmetric continuity of the function $g$ in Theorem \ref{Sum_WS} by weak symmetric continuity. We show this in the next example.

\begin{example}
Let $A = \left\{ \frac{1}{n} \mid n \in \Z- \left\{0\right\} \right\}$ and $B = \left\{ \frac{\sqrt{2}}{n} \mid n \in \Z-\left\{ 0 \right\} \right\}$ . Define $f, g:\R \rightarrow \R$ by
$$
f(x) = \begin{cases}
x, \quad & \text{if} \; x \in A\cup\left\{ 0 \right\}; \\
-1, \quad & \text{if} \; x > 0 \; \text{and} \; x \notin A; \\
1, \quad & \text{if} \; x < 0 \; \text{and} \; x \notin A,  
\end{cases}
\; \text{and} \;
g(x) = \begin{cases}
x, \quad & \text{if} \; x \in B\cup\left\{ 0 \right\}; \\
-2, \quad & \text{if} \; x > 0 \; \text{and} \; x \notin B; \\
2, \quad & \text{if} \; x < 0 \; \text{and} \; x \notin B.  
\end{cases} 
$$
Since $f\left( \frac{1}{n} \right) - f\left(-\frac{1}{n} \right) = \frac{2}{n} \rightarrow 0$ and $g\left( \frac{\sqrt{2}}{n} \right) - g\left(-\frac{\sqrt{2}}{n} \right) = \frac{2\sqrt{2}}{n} \rightarrow 0$ as $n \rightarrow \infty$. So $f$ and $g$ are weakly symmetrically continuous at $0$. We obtain
$$
(f+g)(x) = \begin{cases}
-3, \quad & \text{if} \; x > 0 \; \text{and} \; x \notin A\cup B; \\
3, \quad & \text{if} \; x < 0 \; \text{and} \; x \notin A\cup B;  \\
x-2, \quad & \text{if} \; x > 0 \; \text{and} \; x \in A; \\
x+2, \quad & \text{if} \; x < 0 \; \text{and} \; x \in A; \\
x-1, \quad & \text{if} \; x > 0 \; \text{and} \; x \in B; \\
x+1, \quad & \text{if} \; x < 0 \; \text{and} \; x \in B; \\
0, \quad & \text{if} \; x = 0, 
\end{cases} 
$$
$$
(f-g)(x) = \begin{cases}
1, \quad & \text{if} \; x > 0 \; \text{and} \; x \notin A\cup B; \\
-1, \quad & \text{if} \; x < 0 \; \text{and} \; x \notin A\cup B;  \\
x+2, \quad & \text{if} \; x > 0 \; \text{and} \; x \in A; \\
x-2, \quad & \text{if} \; x < 0 \; \text{and} \; x \in A; \\
-x-1, \quad & \text{if} \; x > 0 \; \text{and} \; x \in B; \\
-x+1, \quad & \text{if} \; x < 0 \; \text{and} \; x \in B; \\
0, \quad & \text{if} \; x = 0, 
\end{cases} 
$$
$$
\max\left\{f,g\right\}(x) = \begin{cases}
-1, \quad & \text{if} \; x > 0 \; \text{and} \; x \notin A\cup B; \\
2, \quad & \text{if} \; x < 0 \; \text{and} \; x \notin A\cup B;  \\
x, \quad & \text{if} \; x > 0 \; \text{and} \; x \in A; \\
2, \quad & \text{if} \; x < 0 \; \text{and} \; x \in A; \\
x, \quad & \text{if} \; x > 0 \; \text{and} \; x \in B; \\
1, \quad & \text{if} \; x < 0 \; \text{and} \; x \in B; \\
0, \quad & \text{if} \; x = 0, 
\end{cases} 
$$
$$
\min\left\{f,g\right\}(x) = \begin{cases}
-2, \quad & \text{if} \; x > 0 \; \text{and} \; x \notin A\cup B; \\
1, \quad & \text{if} \; x < 0 \; \text{and} \; x \notin A\cup B;  \\
-2, \quad & \text{if} \; x > 0 \; \text{and} \; x \in A; \\
x, \quad & \text{if} \; x < 0 \; \text{and} \; x \in A; \\
-1, \quad & \text{if} \; x > 0 \; \text{and} \; x \in B; \\
x, \quad & \text{if} \; x < 0 \; \text{and} \; x \in B; \\
0, \quad & \text{if} \; x = 0. 
\end{cases} 
$$
Let $\left(h_n\right) \in S_0(\R)$. Similar to Example \ref{WC_notimplies_SC_WSC}, for each $n \in \N$,
$$
\left(f+g\right)(h_n) - \left(f+g\right)(-h_n) = \begin{cases}
6, \quad & \text{if} \; h_n > 0 \; \text{and} \; h_n \notin A\cup B; \\ 
2h_n-4, \quad & \text{if} \; h_n > 0 \; \text{and} \; h_n \in A; \\
2h_n-2, \quad & \text{if} \; h_n > 0 \; \text{and} \; h_n \in B.
\end{cases}
$$
So $\left(f+g\right)(h_n) - \left(f+g\right)(-h_n) \in \left\{ 6, 2h_n-4, 2h_n-2 \right\}$ for every $n \in \N$. Since $2h_n-4 \rightarrow -4$ and $2h_n-2 \rightarrow -2$ as $n \rightarrow \infty$, $\left(f+g\right)(h_n) - \left(f+g\right)(-h_n)$ does not converge to zero as $n \rightarrow \infty$. Similarly, we obtain that for every $n \in \N$, 
\begin{align*}
&\left(f-g\right)(h_n) - \left(f-g\right)(-h_n) \in \left\{2,2h_n+4,-2h_n-2\right\}, \\ 
&\max\left\{f,g\right\}(h_n) - \max\left\{f,g\right\}(-h_n) \in \left\{-3, h_n-2, h_n-1\right\}, \quad \text{and} \\ 
&\min\left\{f,g\right\}(h_n) - \min\left\{f,g\right\}(-h_n) \in \left\{-3, -h_n-2, -h_n-1\right\}.
\end{align*}
Therefore they do not converge to zero as $n \rightarrow \infty$. This shows that $f+g$, $f-g$, $\max\left\{f,g\right\}$, and $\min\left\{f,g\right\}$ are not weakly symmetrically continuous at $0$.
\end{example}

Recall that a function $f:A\rightarrow \R$ is said to be locally bounded at $a$ if there exist $\delta > 0$ and $M > 0$ such that $\left|f(x)\right| < M$ for all $x \in \left(a-\delta,a+\delta \right)\cap A$. We have the following result.

\begin{theorem} \label{Product_WS}
Let $f:A\rightarrow \R$ be weakly symmetrically continuous at $a$ and let $g:A\rightarrow \R$ be symmetrically continuous at $a$. If $f$ and $g$ are locally bounded at $a$, then $fg$ is weakly symmetrically continuous at $a$.
\end{theorem}

\begin{proof}
Assume that $f$ and $g$ are locally bounded at $a$ and $S_{a}(A) \neq \emptyset$. Then there exists $\left(h_n\right) \in S_{a}(A)$ such that 
$\lim_{n \rightarrow \infty} \left|f(a+h_n)-f(a-h_n)\right| = 0$. By Lemma \ref{LimitSym}, $\lim_{n \rightarrow \infty} \left|g(a+h_n)-g(a-h_n)\right| = 0$. Since $f$ and $g$ are locally bounded at $a$, there are positive real numbers $M$ and $\delta$ such that 
$ \left| f(x) \right| < M$ and $\left| g(x) \right| < M$, for every $x \in \left(a-\delta , a+\delta \right) \cap A$.
Since $\left(h_n\right) \in S_{a}(A)$, we can choose $N \in \N$ such that $a+h_n$ and $a-h_n$ are in $\left(a-\delta , a+\delta \right) \cap A$ for all $n \geq N$. Writing $(fg)(a+h_n)- (fg)(a-h_n) = f(a+h_n)g(a+h_n)-f(a+h_n)g(a-h_n)+f(a+h_n)g(a-h_n)-f(a-h_n)g(a-h_n)$, we obtain that 
\begin{align*}
&\left| (fg)(a+h_n)- (fg)(a-h_n) \right| \\ 
	&\qquad \leq \left| f(a+h_n)\right| \left|g(a+h_n)-g(a-h_n)\right|+\left|f(a+h_n)-f(a-h_n)\right| \left|g(a-h_n) \right| \\
	&\qquad \leq M \left|g(a+h_n)-g(a-h_n)\right| + M \left|f(a+h_n)-f(a-h_n)\right|, 
\end{align*}
for every $n \geq \N$. This implies that $\lim_{n \rightarrow \infty} \left|(fg)(a+h_n)-(fg)(a-h_n)\right| = 0$. Hence $fg$ is weakly symmetrically continuous at $a$.
\end{proof}

Next we give an example to show that local boundedness of the functions $f$ and $g$ in Theorem \ref{Product_WS} cannot be omitted.

\begin{example} \label{Not_locallybounded}
Define $f,g:\R \rightarrow \R$ by 
$$
f(x)=x \quad \text{and} \quad 
g(x)= \begin{cases}
\frac{1}{\left|x\right|}, \quad & \text{if} \quad x\neq 0; \\
0, \quad & \text{if} \quad x=0. 
\end{cases}
$$
Then $f$ and {g} are symmetrically continuous at $0$, $f$ is locally bounded at $0$ but $g$ is not locally bounded at $0$. By Theorem \ref{SC_implie_WSC}, $f$ is weakly symmetrically continuous at $0$. In addition, 
$$
\left(fg\right)(x)= \begin{cases}
1, \quad & \text{if} \quad x>0; \\
-1, \quad & \text{if} \quad x<0; \\
0, \quad & \text{if} \quad x=0. 
\end{cases}
$$
Therefore, for every $\left(h_n\right) \in S_0(\R)$, $\left(fg\right)(h_n) - \left(fg\right)(- h_n) = 2 \nrightarrow 0$ as $n \rightarrow \infty$. Hence, $fg$ is not weakly symmetrically continuous at $0$.
\end{example}
  
Next we give functions $f$ and $g$ such that $f$ and $g$ are weakly symmetrically continuous at $0$ and locally bounded at $0$ but $fg$ is not weakly symmetrically continuous at $0$. Therefore the symmetric continuity of the function $g$ in Theorem \ref{Product_WS} cannot be replaced by weak symmetric continuity.

\begin{example}
Let $A = \left\{ \frac{1}{n} \mid n \in \Z- \left\{0\right\} \right\}$, $B = \left\{ \frac{\sqrt{2}}{n} \mid n \in \Z-\left\{ 0 \right\} \right\}$, and $C = A\cup B\cup\left\{0\right\}$. Define $f, g:C \rightarrow \R$ by
$$
f(x) = \begin{cases}
1, \quad & \text{if} \; x \in B\cup\left\{ 0 \right\}; \\
\frac{1}{x^2+1}, \quad & \text{if} \; x > 0 \; \text{and} \; x \in A; \\
-\frac{1}{x^2+1}, \quad & \text{if} \; x < 0 \; \text{and} \; x \in A,  
\end{cases}
$$
and
$$
g(x) = \begin{cases}
1, \quad & \text{if} \; x \in A\cup\left\{ 0 \right\}; \\
\frac{1}{x^2+2}, \quad & \text{if} \; x > 0 \; \text{and} \; x \in B; \\
-\frac{1}{x^2+2}, \quad & \text{if} \; x < 0 \; \text{and} \; x \in B.  
\end{cases} 
$$
Observe that $f\left(\frac{\sqrt{2}}{n}\right) - f\left(- \frac{\sqrt{2}}{n}\right) = 0$ and $g\left(\frac{1}{n}\right) - g\left(- \frac{1}{n}\right) = 0$ for every $n \in \N$. Therefore $f$ and $g$ are weakly symmetrically continuous at $0$. It is easy to see that $f$ and $g$ are bounded. So $f$ and $g$ are locally bounded at $0$. However, 
$$
\left(fg\right)(x) = \begin{cases}
\frac{1}{x^2+1}, \quad & \text{if} \; x > 0 \; \text{and} \; x \in A; \\
-\frac{1}{x^2+1}, \quad & \text{if} \; x < 0 \; \text{and} \; x \in A; \\
\frac{1}{x^2+2}, \quad & \text{if} \; x > 0 \; \text{and} \; x \in B; \\
-\frac{1}{x^2+2}, \quad & \text{if} \; x < 0 \; \text{and} \; x \in B; \\
1, \quad & \text{if} \; x=0.
\end{cases}
$$
So for every $\left(h_n\right) \in S_0(\R)$, we have $\left(fg\right)(h_n) - \left(fg\right)(- h_n) \in \left\{\frac{2}{\left(h_n\right)^2+1}, \frac{2}{\left(h_n\right)^2+2}\right\}$ for every $n \in \N$. So $\left(\left(fg\right)(h_n) - \left(fg\right)(- h_n)\right)$ does not converge to $0$ as $n \rightarrow \infty$. Therefore $fg$ is not weakly symmetrically continuous at $0$.
\end{example}

Next we study symmetric continuity, weak symmetric continuity and the quotient of functions.

\begin{theorem}
Let $f:A\rightarrow \R$ be weakly symmetrically continuous at $a$. Assume that $f(x)\neq 0$ for all $x\in A$ and $\frac{1}{f}$ is locally bounded at $a$. Then $\frac{1}{f}$ is weakly symmetrically continuous at $a$.
\end{theorem}

\begin{proof}
Assume that $S_{a}(A) \neq \emptyset$. Then there exists $\left(h_n\right) \in S_{a}(A)$ such that $\lim_{n \rightarrow \infty} \left|f(a+h_n)-f(a-h_n)\right| = 0$. Since $\frac{1}{f}$ is locally bounded at $a$, there are positive real numbers $M$ and $\delta$ such that 
\begin{align*}
	\left| \frac{1}{f(x)} \right| < M \quad \text{for every $x \in \left(a-\delta , a+\delta \right) \cap A$.}
\end{align*}
Since $\left(h_n\right) \in S_{a}(A)$, we can choose $N \in \N$ such that $a+h_n$ and $a-h_n$ are in $\left(a-\delta , a+\delta \right) \cap A$  for all $n \geq N$. Then for every $n \geq \N$,
\begin{align*}
 \left| \frac{1}{f(a+h_n)} - \frac{1}{f(a-h_n)} \right|  =  \left| \frac{f(a+h_n) - f(a-h_n)}{f(a+h_n)f(a-h_n)} \right|  \leq  M^{2}\left|f(a+h_n) - f(a-h_n)\right|. 
\end{align*} 
This implies that $\lim_{n \rightarrow \infty} \left| \frac{1}{f(a+h_n)} - \frac{1}{f(a-h_n)} \right| = 0$. Hence $\frac{1}{f}$ is weakly symmetrically continuous at $a$.
\end{proof}

\begin{lemma} \cite[Lemma~2.3]{R_Sym} \label{LS}
Let $f:A\rightarrow \R$ be symmetrically continuous at $a$ and $f(x) \neq 0$ for every $x\in A$. If $\frac{1}{f}$ is locally bounded at $a$, then $\frac{1}{f}$ is symmetrically continuous at $a$.
\end{lemma}

\begin{theorem}
Let $f:A\rightarrow \R$ be weakly symmetrically continuous at $a$ and locally bounded at $a$. Let $g:A\rightarrow \R$ be symmetrically continuous at $a$. If $g(x)\neq 0$ for all $x \in A$ and $\frac{1}{g}$ is locally bounded at $a$, then $\frac{f}{g}$ is weakly symmetrically continuous at $a$.
\end{theorem}

\begin{proof}
This follows immediately from Lemma \ref{LS} and Theorem \ref{Product_WS}.
\end{proof}

Next we obtain a result on the composition of functions.

\begin{theorem} \label{Composit_WS}
Let $f:A\rightarrow B$ and $g:B \rightarrow \R$. Assume that $f$ is weakly symmetrically continuous at $a$ and $g$ is uniformly continuous on $B$. Then $g\circ f$ is weakly symmetrically continuous at $a$.
\end{theorem}

\begin{proof}
Assume that $S_{a}(A) \neq \emptyset$. Then there exists $\left(h_n\right) \in S_{a}(A)$ such that $\lim_{n \rightarrow \infty} \left|f(a+h_n)-f(a-h_n)\right| = 0$. Let $\varepsilon > 0$. Then there exists $\delta > 0$ such that for all $x,y \in B$, if $\left|x-y\right| < \delta$, then 
\begin{align} \label{Eqn1}
	\left|g(x)-g(y)\right| < \varepsilon.
\end{align}
There is $N \in \N$ such that for all $n \geq \N$
\begin{align} \label{Eqn2}
	\left|f(a+h_n)-f(a-h_n)\right| < \delta.
\end{align}
Let $n \geq \N$. By $\left(\ref{Eqn1}\right)$ and $\left(\ref{Eqn2}\right)$, we obtain
\begin{align*}
	\left|(g\circ f)(a+h_n)-(g\circ f)(a-h_n)\right| = \left|g(f(a+h_n))-g(f(a-h_n)) \right| < \varepsilon.
\end{align*}
This shows that $(g\circ f)(a+h_n)-(g\circ f)(a-h_n)$ converges to $0$ as $n \rightarrow \infty$, and therefore $g\circ f$ is weakly symmetrically continuous at $a$. 
\end{proof}

The uniform continuity of the function $g$ in Theorem \ref{Composit_WS} cannot be replaced by continuity as shown in the next example.
\begin{example}
Defined $f, g:\R \rightarrow \R$ by
$$
f(x) = \begin{cases}
x+\frac{1}{x}, \quad & \text{if} \quad  x > 0; \\
-\frac{1}{x}, \quad & \text{if} \quad  x < 0;  \\
0, \quad & \text{if} \quad x = 0,   
\end{cases} 
\quad \text{and} \quad
g(x) =  \begin{cases}
x^2, \quad & \text{if} \quad  x \geq 0; \\
x, \quad & \text{if} \quad  x < 0.   
\end{cases} 
$$
It is shown in \cite{R_Sym} that $f$ is symmetrically continuous at $0$ and $g$ is continuous but is not uniformly continuous on $\R$. By Theorem \ref{SC_implie_WSC}, $f$ is weakly symmetrically continuous at $0$. In addition,
$$
\left(g\circ f\right)(x) = \begin{cases}
x^2+\frac{1}{x^2}+2 , \quad & \text{if} \quad  x > 0;  \\
\frac{1}{x^2}, \quad & \text{if} \quad  x < 0;  \\
0, \quad & \text{if} \quad x = 0.   
\end{cases} 
$$
Then for every $(h_n) \in S_0(\R)$, $\left(g\circ f\right)(0+h_n) - \left(g\circ f\right)(0-h_n) = (h_n)^2+\frac{1}{(h_n)^2}+2-\frac{1}{(h_n)^2} = (h_n)^2 + 2 \rightarrow 2 \neq 0$ as $n \rightarrow \infty$. So $g\circ f$ is not weakly symmetrically continuous at $0$.
\end{example}

\begin{corollary}
Let $f:A\rightarrow \R$ be weakly symmetrically continuous at $a \in A$ and $f(x) \geq 0 $ for all $x \in A$. Then $\sqrt{f}$ is weakly symmetrically continuous at $a$ 
\end{corollary} 

\begin{proof} 
Let $g:[0, \infty) \rightarrow \R$ be defined by $g(x)=\sqrt{x}$. Then $g$ is uniformly continuous on $[0,\infty)$ and $g\circ f=\sqrt{f}$. By Theorem \ref{Composit_WS}, $\sqrt{f}$ is weakly symmetrically continuous at $a$. 
\end{proof}

\section{Sequences of Functions}

Let $\left(f_k\right)$ be a sequence of weakly symmetrically continuous functions.
In this section, we obtain basic results on pointwise and uniform convergence of $\left(f_k\right)$. We will see that pointwise limit of $\left(f_k\right)$ may or may not be weakly symmetrically continuous. In fact, although $f_k$ is symmetrically continuous for every $k$, the pointwise limit of $\left(f_k\right)$ may or may not be weakly symmetrically continuous as shown in the next example.

\begin{example}
For each $k \in \N$, let $f_k:\left[0,2\right] \rightarrow \R$ be given by 
$$
f_k(x) = \begin{cases}
x^k, \quad & \text{if} \quad x \in \left[0,1\right]; \\
1, \quad & \text{if} \quad x > 1.
\end{cases} 
$$
Let $f:\left[0,2\right] \rightarrow \R$ be given by
$$
f(x) = \begin{cases}
0, \quad & \text{if} \quad x \in \left[0,1\right);  \\
1, \quad & \text{if} \quad x \geq 1. 
\end{cases} 
$$
Then $f_k$ is a symmetrically continuous function for every $k$, $f_k \rightarrow f$, but $f$ is not weakly symmetrically continuous at $1$. 
\end{example}

Next we show that uniform limit of weakly symmetrically continuous functions is weakly symmetrically continuous. A method similar to Cantor diagonalization process is used in the proof.

\begin{theorem}
For each $k \in \N$, let $f_k:A \rightarrow \R$ be weakly symmetrically continuous at $a$. Suppose that $\left(f_k\right)$ converges uniformly on $A$ to a function $f$. Then $f$ is weakly symmetrically continuous at $a$.
\end{theorem}

\begin{proof}
Assume that $S_{a}(A) \neq \emptyset$. For each $k \in \N$, since $f_k$ is weakly symmetrically continuous at $a$, there exists a sequence $\left(h_{n}^{\left(k\right)}\right)_{n \in \N} \in S_{a}(A)$ such that $\lim_{n \rightarrow \infty} \left|f_k(a+h_{n}^{\left(k\right)})-f_k(a-h_{n}^{\left(k\right)})\right| = 0$.	
For each $k \in \N$, there exists $N_k \in \N$ such that
\begin{equation} \label{The5eqn1}
 \left| h_{N_k}^{\left(k\right)} \right| < \frac{1}{k} \quad \text{and} \quad \left|f_k(a+h_{N_k}^{\left(k\right)})-f_k(a-h_{N_k}^{\left(k\right)})\right| < \frac{1}{k}. 
\end{equation}
For each $n \in \N$, let $x_n = h_{N_n}^{\left(n\right)}$.
Then $\left(x_n\right)$ converges to zero and $x_n > 0$ for every $n \in \N$. In addition, $a+x_n$ and $a-x_n$ are in $A$ for every $n \in \N$. Therefore $\left(x_n\right) \in S_a(A)$. It remains to show that $\lim_{n \rightarrow \infty} \left|f(a+x_n) - f(a-x_n)\right| = 0$.
Let $\varepsilon > 0$. Since $\left(f_k\right)$ converges uniformly to $f$ on $A$, there exists $M \in \N$ such that
\begin{equation} \label{The5eqn2}
	 \left|f_k(x)-f(x)\right| < \frac{\varepsilon}{3}
\end{equation}
for all $k \geq M$ and $x \in A$. There exists $K \in \N$ such that 
\begin{equation} \label{The5eqn3}
\frac{1}{K} < \frac{\varepsilon}{3}.
\end{equation} 
Let $N=\max\left\{M,K\right\}$ and let $n \geq N$. Then By $\left(\ref{The5eqn2}\right)$, $\left(\ref{The5eqn1}\right)$, the definition of $x_n$, and $\left(\ref{The5eqn3}\right)$, we obtain that 
\begin{align*}
	& \left|f(a+x_n)-f(a-x_n)\right| \\
	& \leq \left|f(a+x_n)-f_n(a+x_n)\right|+\left|f_n(a+x_n)-f_n(a-x_n)\right|+\left|f_n(a-x_n)-f(a-x_n)\right| \\
	& < \frac{\varepsilon}{3} + \left|f_n(a+x_n)-f_n(a-x_n)\right| + \frac{\varepsilon}{3} < \frac{\varepsilon}{3} + \frac{1}{n}+ \frac{\varepsilon}{3} < \frac{\varepsilon}{3} + \frac{\varepsilon}{3}+ \frac{\varepsilon}{3} = \varepsilon. 
\end{align*}
This completes the proof.
\end{proof}

\vskip.5cm \noindent{\bf Acknowledgement(s) :} The first author would like to thank Department of Mathematics, Faculty of Science, Silpakorn University for the support. The second author would like to thank the Development and Promotion of Science and Technology Talents Project (DPST) of Thailand for giving him scholarship.

\end{document}